\documentclass[11pt]{article}

\usepackage{amsmath,amssymb,amsthm}
\usepackage{algorithm,algorithmic}
\usepackage{tikz}
\usepackage{hyperref}
\usepackage{mathtools}
\usepackage{xcolor}
\usepackage{lineno}

\newtheorem{theorem}{Theorem}[section]
\newtheorem{lemma}[theorem]{Lemma}
\newtheorem{proposition}[theorem]{Proposition}
\newtheorem{corollary}[theorem]{Corollary}

\theoremstyle{definition}
\newtheorem{definition}[theorem]{Definition}
\newtheorem{remark}[theorem]{Remark}
\newtheorem{example}[theorem]{Example}

\newcommand{\F}{\mathbb{F}}

\newcommand{\Fcal}{\mathcal{F}}
\newcommand{\Gam}{\Gamma}
\newcommand{\GL}{\mathrm{GL}}
\newcommand{\AGL}{\mathrm{AGL}}
\newcommand{\id}{\mathrm{id}}
\newcommand{\rank}{\operatorname{rank}}

\newcommand{\Hom}{\mathrm{Hom}}
\newcommand{\Stab}{\mathrm{Stab}}

\newcommand{\Fix}{\mathrm{Fix}}
\renewcommand{\Pr}{\mathbb{P}}

\title{Almost All Vectorial Functions Have Trivial Extended-Affine Stabilizers}
\author{Keita Ishizuka\\
Information Technology R\&D Center\\
Mitsubishi Electric Corporation\\
Kanagawa, Japan\\
\texttt{keitaishizuka1994@gmail.com}}
\date{}

\begin{document}
\maketitle

\begin{abstract}
We prove that asymptotically almost all vectorial functions over finite fields have trivial extended-affine stabilizers. As a consequence, the number of EA-equivalence classes is asymptotically equal to the naive estimate, namely the total number of functions divided by the size of the EA-group, with vanishing relative error.
Furthermore, we derive upper bounds on collision probabilities for both extended-affine and CCZ equivalences. For EA-equivalence, we leverage the trivial-stabilizer result to establish a matching lower bound, yielding a tight asymptotic formula that shows two independently sampled functions are EA-equivalent with super-exponentially small probability.
The results validate random sampling strategies for cryptographic primitive design and show that functions with nontrivial EA-stabilizers form an exponentially rare subset.
\end{abstract}


\section{Introduction}

Vectorial Boolean functions play a crucial role in symmetric
cryptography, serving as S-boxes in block ciphers and hash functions. The classification
of these functions is typically studied modulo certain equivalence relations that preserve
important cryptographic properties. Extended-affine (EA) equivalence is one of the most
fundamental such relations: it preserves differential uniformity, the Walsh spectrum, the
nonlinearity, and the algebraic degree. A closely related but strictly coarser notion is
CCZ-equivalence, introduced by Carlet, Charpin, and Zinoviev~\cite{carlet1998}, which
preserves differential uniformity, the Walsh spectrum, and the nonlinearity but does not
in general preserve the algebraic degree~\cite{budaghyan2020}.

Two fundamental questions arise in the study of EA-equivalence:
\begin{enumerate}
\item \emph{Classification:} How many EA-equivalence classes are there? Since EA-equivalent
functions share the same cryptographic properties, counting equivalence classes determines the
number of essentially distinct functions relevant for symmetric cryptography. This is crucial
for understanding the diversity of cryptographic primitives.

\item \emph{Random sampling:} When searching for functions with desirable properties via random generation, is there
a risk of repeatedly sampling EA-equivalent functions? In other words, does random search remain
effective modulo EA-equivalence, or do collision events significantly limit the exploration of
the equivalence class space?
\end{enumerate}

Both questions are intimately connected to the structure of EA-stabilizer groups.
The EA-stabilizer of $F$ is the set of all triples $(A_{\rm in}, A_{\rm out}, L)$
satisfying $A_{\rm out}\circ F\circ A_{\rm in}+L = F$;
it measures how much ``self-equivalence'' the function possesses.
By the orbit-stabilizer theorem, the size of each equivalence class is inversely
proportional to the stabilizer size, which directly impacts both the total number of
classes and the collision probability. However, explicit computation of stabilizers is
infeasible except for very small dimensions, as the function space grows doubly
exponentially with the input dimension.

\subsection{Our Contributions}

The central contribution of this paper is to establish, for vectorial functions over
finite fields, that the EA-group acts \emph{asymptotically freely}: almost every function
has trivial EA-stabilizer. From this single structural statement, both questions above
are answered in tight asymptotic form, and a previously unknown matching lower bound on
the EA-collision probability follows. The contributions are as follows.
\begin{enumerate}
\item \emph{Trivial EA-stabilizers} (Theorem~\ref{thm:trivial-stab}): for a uniformly random
$F\in\Fcal$, $\Pr(\Stab_\Gam(F)\neq\{\id\})=q^{-\Omega(mq^n)}$. The proof is
elementary: it bounds $|\Fix(g)|$ for each nontrivial $g\in\Gam$ by analysing the orbit
structure of affine permutations and applies a union bound. Restricted to ordinary
(non-vectorial) Boolean functions and the action of $\GL(n,2)$, this recovers the
trivial-symmetry theorems of Clausen~\cite{clausen1992} and Clote and
Kranakis~\cite{clote1991}.

\item \emph{Probabilistic proof of the asymptotic EA-orbit count}
(Corollary~\ref{cor:orbit-count}): the number of EA-equivalence classes satisfies
$|\Fcal/\Gam|=(|\Fcal|/|\Gam|)\,(1+o(1))$. The same asymptotic was proved
analytically by Hou~\cite{hou2021} (for $\F_2^n\to\F_2$) and by
Lu et al.~\cite{lu2023} (for vectorial Boolean functions) using
Burnside's lemma combined with compound-matrix rank estimates and M\"obius inversion.
Our argument follows a fundamentally different route: the asymptotic is a one-line
consequence of Theorem~\ref{thm:trivial-stab} via the orbit-stabilizer theorem,
exposing the structural reason behind the naive estimate
(trivial stabilizers with high probability) in place of analytic group-character
computations.

\item \emph{Tight asymptotic collision probability} (Proposition~\ref{cor:collision-upper}
and Corollary~\ref{cor:collision-ea-asymptotic}): we give upper bounds on the
collision probability for both EA- and CCZ-equivalence which follow from
orbit-stabilizer alone.
Furthermore, by using Theorem~\ref{thm:trivial-stab}, we prove a
\emph{matching lower bound} for EA-equivalence, yielding
$\Pr(F\sim_{\mathrm{EA}} G)=(|\Gam|/|\Fcal|)(1+o(1))$. This sharp two-sided
asymptotic for the EA-collision probability does not follow from
Hou~\cite{hou2021}, Lu et al.~\cite{lu2023}, or earlier work, and is genuinely new.
\end{enumerate}

\subsection{Related Work}
We organize prior work into three categories:

\emph{(i) Trivial-symmetry theorems for Boolean functions.}
Clausen~\cite{clausen1992} proved that almost all Boolean functions $\F_2^n\to\F_2$ have
trivial $\GL(n,2)$-stabilizer; Clote and Kranakis~\cite{clote1991} obtained an
analogous statement for invariance under coordinate permutations. Both results are
restricted to the codomain $\F_2$ and to a substantially smaller group than $\Gam$.
Theorem~\ref{thm:trivial-stab} extends this line to arbitrary vectorial codomains
$\F_q^m$ and to the full EA-group, including the linear summand $L\in\Hom(U,W)$.

\emph{(ii) Counting EA-classes.}
Hou~\cite{hou2021} obtained an asymptotic formula for the number of EA-equivalence
classes of Boolean functions using Burnside's lemma and compound-matrix rank estimates.
Lu et al.~\cite{lu2023} extended this to vectorial Boolean functions via M\"obius
inversion combined with Hou's matrix-theoretic apparatus. Our
Corollary~\ref{cor:orbit-count} matches these asymptotics for vectorial Boolean
functions and extends them to arbitrary $q$, with a shorter proof that does not require
analytic estimates on compound matrices.

\emph{(iii) Structure of EA- and CCZ-equivalence.}
A complementary line of work studies the relationship between EA- and CCZ-equivalence
and develops algorithms or invariants for distinguishing classes:
Canteaut and Perrin~\cite{canteaut2019} showed that CCZ-equivalence decomposes as
EA-equivalence followed by an operation called \emph{twisting};
Budaghyan, Calderini, and Villa~\cite{budaghyan2020} gave sufficient conditions for
CCZ-equivalence to be strictly more general than EA-equivalence combined with inverses;
Kaleyski~\cite{kaleyski2021} surveyed EA/CCZ invariants for APN and AB functions; and
Canteaut, Couvreur, and Perrin~\cite{canteaut2022} gave an efficient algorithm for
recovering an EA-transformation between two quadratic functions.
These works address \emph{testing} EA-equivalence between explicit pairs of functions
and are largely orthogonal to the present asymptotic-enumeration question.

\subsection{Paper Organization}

Section~\ref{sec:prelim} reviews necessary
background on group actions, Boolean functions, and EA-equivalence. Section~\ref{sec:main} presents
our main results, including the probabilistic analysis of EA-stabilizers and the asymptotic
count of EA-classes. Section~\ref{sec:conclusion} concludes with a discussion of implications
and open problems.

\section{Preliminaries}\label{sec:prelim}

\subsection{Group Actions and Probabilistic Tools}

Let $q$ be a prime power and let $\F_q$ denote the finite field with $q$ elements.
Let $U=\F_q^n$ and $W=\F_q^m$ denote finite vector spaces over $\F_q$.
The \emph{general linear group} $\GL(n,q)$ consists of all invertible $n\times n$
matrices over $\F_q$, and has order
\[
   |\GL(n,q)| = \prod_{i=0}^{n-1}(q^n-q^i).
\]

\begin{definition}[Affine general linear group]
The \emph{affine general linear group} $\AGL(U)$ is the semidirect product
$U\rtimes\GL(n,q)$, acting on $U$ by $\sigma(x)=Px+a$ with $P\in\GL(n,q)$ and $a\in U$.
Its order is $|\AGL(U)|=q^n\cdot|\GL(n,q)|$.
\end{definition}

\begin{definition}[Orbits and stabilizers]
Let $G$ be a group acting on a set $X$. For $x\in X$, the \emph{$G$-orbit} of $x$ is
\[
   G\cdot x := \{g\cdot x : g\in G\}.
\]
The \emph{stabilizer} of $x$ is
\[
   \Stab_G(x):=\{g\in G: g\cdot x = x\}.
\]
For each $g\in G$, the \emph{fixed-point set} is
\[
   \Fix(g):=\{x\in X: g\cdot x=x\}.
\]
\end{definition}

\begin{lemma}[Orbit-stabilizer theorem]
For any $x\in X$, the size of the orbit $G\cdot x$ satisfies
\[
   |G\cdot x| = \frac{|G|}{|\Stab_G(x)|}.
\]
\end{lemma}

\begin{lemma}[Burnside's lemma]
The number of $G$-orbits is given by
\[
   |X/G| = \frac{1}{|G|}\sum_{g\in G}|\Fix(g)|.
\]
\end{lemma}

We use standard asymptotic notation: for functions $f,g:\mathbb{N}\to\mathbb{R}_{>0}$,
we write $f(n)=O(g(n))$ if there exist constants $C>0$ and $n_0$ such that
$f(n)\le C\cdot g(n)$ for all $n\ge n_0$. We write $f(n)=\Omega(g(n))$ if
$g(n)=O(f(n))$, and $f(n)=o(g(n))$ if $\lim_{n\to\infty}f(n)/g(n)=0$.

\begin{lemma}[Union bound]
If $E_1,\dots,E_k$ are events, then
\[
   \Pr\Bigl(\bigcup_{i=1}^k E_i\Bigr) \le \sum_{i=1}^k \Pr(E_i).
\]
\end{lemma}

\subsection{Vectorial Functions and Equivalence Relations}

\begin{definition}[Vectorial functions]
We write $\Fcal=\{F:U\to W\}$ for the set of all functions from $U$ to $W$.
Note that $|\Fcal|=q^{mq^n}$.
\end{definition}

When $q=2$, functions in $\Fcal$ are called \emph{vectorial Boolean functions}.
They play a central role in symmetric cryptography, where they are used as
S-boxes in block ciphers and hash functions.
The cryptographic quality of a vectorial function is measured by several properties.
\emph{Differential uniformity} quantifies resistance to differential cryptanalysis and is
defined as the maximum number of solutions to $F(x+a) - F(x) = b$ over all nonzero $a$ and all
$b$. The \emph{Walsh spectrum} and \emph{nonlinearity} measure resistance to linear cryptanalysis,
while the \emph{algebraic degree} controls the complexity of algebraic attacks. Functions
optimizing these properties, such as almost perfect nonlinear (APN) and almost bent (AB)
functions, are of particular interest. EA-equivalence preserves all of the above
properties, while CCZ-equivalence preserves differential uniformity, the Walsh spectrum,
and nonlinearity, but does \emph{not} in general preserve the algebraic
degree~\cite{budaghyan2020}. For further details on these properties and their
cryptanalytic significance, we refer to~\cite{carlet2010}.

\begin{definition}[Extended-affine equivalence]\label{def:EA}
The \emph{extended-affine (EA) group}
\[
   \Gam := \AGL(U)\times \AGL(W)\times \Hom(U,W)
\]
acts on $\Fcal$ as follows: for $g=(A_{\rm in},A_{\rm out},L)\in \Gam$ with
$A_{\rm in}\in\AGL(U)$, $A_{\rm out}\in\AGL(W)$, and $L\in\Hom(U,W)$,
\[
   (g\cdot F)(x)
   := A_{\rm out}\bigl(F(A_{\rm in}(x))\bigr)+L(x).
\]
Explicitly, if $A_{\rm in}(x)=Px+a$ and $A_{\rm out}(y)=Qy+b$, then
\[
   \bigl((P,a,Q,b,L)\cdot F\bigr)(x) = Q\,F(Px+a)+Lx+b.
\]
Two functions $F,G\in\Fcal$ are \emph{EA-equivalent}, denoted by $F \simeq_{\mathrm{EA}} G$, if they lie in the same
$\Gam$-orbit. The \emph{EA-stabilizer} of $F$ is $\Stab_\Gam(F)$.
The order of the EA group is
\[
   |\Gam|
   = |\AGL(U)|\cdot|\AGL(W)|\cdot q^{nm}
   = q^{n^2+m^2+nm+O(n+m)}.
\]
\end{definition}

The linear summand $L\in\Hom(U,W)$ in Definition~\ref{def:EA} is essential:
without it the action reduces to the (strictly smaller) action of
$\AGL(U)\times\AGL(W)$ by pre- and post-composition, sometimes called
\emph{affine equivalence}~\cite{carlet2010}.

\begin{definition}[Graph of a function]
For a function $F: U \to W$, the \emph{graph} of $F$ is the subset
\[
   \mathrm{Gr}(F) := \{(x,F(x)) : x \in U\} \subseteq U \times W = \F_q^{n+m}.
\]
\end{definition}

\begin{definition}[CCZ-equivalence]
Two functions $F,G: U \to W$ are \emph{CCZ-equivalent} (Carlet-Charpin-Zinoviev equivalent)
if their graphs $\mathrm{Gr}(F)$ and $\mathrm{Gr}(G)$ are affine equivalent as subsets of $U \times W$.
That is, $F$ and $G$ are CCZ-equivalent, denoted by $F \simeq_{\mathrm{CCZ}} G$, if there exists
$\Phi \in \AGL(U \times W)$ such that $\Phi(\mathrm{Gr}(F)) = \mathrm{Gr}(G)$.
\end{definition}

CCZ-equivalence is strictly more general than EA-equivalence: every pair of EA-equivalent
functions is also CCZ-equivalent.

\section{Main Results}\label{sec:main}

\subsection{Fixed-Point and Orbit Structure of Affine Permutations}

Before analysing EA-stabilizers, we establish key properties of affine permutations
over $\F_q$ that will control the orbit structure.

\begin{lemma}[Fixed-point structure]\label{lem:affine-fixed}
Let $\sigma(x)=Px+a$ be an affine permutation of $U=\F_q^n$ with $(P,a)\neq(I,0)$.
Then the fixed-point set
\[
   \{x\in U:\sigma(x)=x\}
\]
is either empty or an affine subspace of dimension at most $n-1$.
In particular, $|\{x\in U:\sigma(x)=x\}|\le q^{n-1}$.
\end{lemma}

\begin{proof}
The fixed-point equation $Px+a=x$ is equivalent to $(P-I)x=-a$.
Since this is a system of linear equations, the solution set (if nonempty) is an
affine subspace. If $P\neq I$, then $\rank(P-I)\ge 1$, so the solution space
has dimension at most $n-1$. If $P=I$ and $a\neq 0$ (pure translation), there are
no fixed points.
\end{proof}

\begin{lemma}[Orbit count]\label{lem:orbit-count}
Let $\sigma$ be a permutation of a finite set $X$ of size $N$.
A point $x\in X$ is a \emph{fixed point of $\sigma$} if $\sigma(x)=x$, and an
\emph{orbit of $\sigma$} is an orbit of the cyclic group $\langle\sigma\rangle$ acting
on $X$. Let $f$ denote the number of fixed points of $\sigma$ and $s$ the number of its
orbits. Then
\[
   s\le \frac{N+f}{2}.
\]
In particular, if $\sigma(x)=Px+a$ is a nontrivial affine permutation of $U=\F_q^n$,
then the number of orbits satisfies
\[
   s\le \frac{q^n+q^{n-1}}{2} = \frac{q+1}{2q}\cdot q^n.
\]
\end{lemma}

\begin{proof}
The $f$ fixed points each form singleton orbits. The remaining $N-f$ non-fixed points
lie in orbits of size at least $2$, so they contribute at most $(N-f)/2$ orbits.
Hence
\[
   s\le f + \frac{N-f}{2} = \frac{N+f}{2}.
\]
For a nontrivial affine permutation, Lemma~\ref{lem:affine-fixed} gives $f\le q^{n-1}$,
so
\[
   s \le \frac{q^n+q^{n-1}}{2}
   = \frac{q^{n-1}(q+1)}{2}
   = \frac{q+1}{2q}\cdot q^n.
\]
\end{proof}

\subsection{Probabilistic Analysis of EA-Stabilizers}

We now analyse the stabilizer distribution and deduce
precise asymptotics for the number of EA-classes.

\begin{lemma}[Fixed-point bound for EA-group elements]\label{lem:fix-g}
Let $g=(P,a,Q,b,L)\in \Gam$ be nontrivial, where $P\in\GL(n,q)$, $a\in U$,
$Q\in\GL(m,q)$, $b\in W$, and $L\in\Hom(U,W)$.
Then the number of $F\in\Fcal$ satisfying $g\cdot F = F$ is bounded by
\[
   |\Fix(g)| \;\le\; q^{cmq^n}
\]
for some constant $c<1$ independent of $n,m$.
\end{lemma}

\begin{proof}
The fixed-point condition $g\cdot F=F$ reads
\begin{equation}\label{eq:fixed-eq}
   Q\,F(Px+a) + Lx + b = F(x)
   \quad(\forall x\in U).
\end{equation}
We distinguish two cases.

\medskip
\noindent\textbf{Case~1: $(P,a)\neq(I,0)$.}
Let $\sigma(x):=Px+a$; by assumption $\sigma\neq\id$.
Lemma~\ref{lem:orbit-count} shows that $\sigma$ has at most
\[
   s\le\frac{q+1}{2q}\cdot q^n
\]
orbits.
Decompose $U$ into $\sigma$-orbits:
\[
   U = \bigsqcup_{i=1}^s O_i,
   \qquad O_i = \{x_i,\sigma(x_i),\sigma^2(x_i),\dots,\sigma^{\ell_i-1}(x_i)\},
\]
where $\ell_i\ge 1$ is the length of orbit $O_i$.

Fix an orbit $O_i$ of length $\ell_i$.
Iterating~\eqref{eq:fixed-eq} starting from $x_i$ gives
\begin{align*}
   F(x_i)
      &= Q\,F(\sigma(x_i)) + Lx_i + b \\
      &= Q^2 F(\sigma^2(x_i)) + Q\bigl(L\sigma(x_i)+b\bigr) + Lx_i + b \\
      &\;\;\vdots \\
      &= Q^{\ell_i}F(x_i) + \sum_{j=0}^{\ell_i-1}Q^j\bigl(L\sigma^j(x_i)+b\bigr).
\end{align*}
Setting $r_i:=\sum_{j=0}^{\ell_i-1}Q^j\bigl(L\sigma^j(x_i)+b\bigr)\in W$ (which is
determined once $g$ and the orbit representative $x_i$ are fixed), this implies
\begin{equation}\label{eq:orbit-constraint}
   \bigl(I - Q^{\ell_i}\bigr) F(x_i) = r_i.
\end{equation}
The number of solutions $F(x_i)\in W$ to~\eqref{eq:orbit-constraint} is at most
$q^{\dim\ker(I-Q^{\ell_i})}\le q^m$.
Once $F(x_i)$ is chosen, the values $F(\sigma^k(x_i))$ for $1\le k<\ell_i$ are
uniquely determined by~\eqref{eq:fixed-eq}.

Repeating this argument for each orbit, we see that $F$ is determined by
choosing at most one free parameter (of size $\le q^m$) per orbit.
Since $s\le (q+1)/(2q)\cdot q^n$, the total number of admissible $F$ is at most
\[
   |\Fix(g)| \le (q^m)^s \;\le\; q^{m\cdot\frac{q+1}{2q}\cdot q^n}
   = q^{\frac{q+1}{2q}mq^n}.
\]
Hence $c=(q+1)/(2q)$ suffices in this case.

\medskip
\noindent\textbf{Case~2: $(P,a)=(I,0)$ but $(Q,b,L)\neq(I,0,0)$.}
Now~\eqref{eq:fixed-eq} becomes
\[
   Q\,F(x) + Lx + b = F(x)
   \quad\Longleftrightarrow\quad
   (I-Q)F(x) = Lx + b
   \qquad(\forall x\in U).
\]
If $Q\neq I$, then $\rank(I-Q)\ge 1$, so $\dim\ker(I-Q)\le m-1$, and for each
$x$ the equation $(I-Q)F(x)=Lx+b$ has at most $q^{m-1}$ solutions. Hence
\[
   |\Fix(g)| \;\le\; (q^{m-1})^{q^n} = q^{(m-1)q^n}.
\]
If $Q=I$, then the equation reduces to $Lx+b=0$ for all $x\in U$, which forces
$L=0$ and $b=0$, contradicting nontriviality. So no $F$ satisfies the equation,
and $|\Fix(g)|=0$.
In either subcase, $c = 1 - 1/m$ works when $m\ge 1$.

Combining both cases, we may take $c=\max\bigl\{(q+1)/(2q),1-1/m\bigr\}<1$.
\end{proof}

\begin{theorem}[Almost all functions have trivial stabilizers]\label{thm:trivial-stab}
For a uniformly random $F\in\Fcal$,
\[
  \Pr\bigl(\Stab_\Gam(F)\neq\{\id\}\bigr)
  \;\le\;
  |\Gam|\cdot q^{-(1-c)mq^n}
  \;=\;
  q^{-\Omega(mq^n)}.
\]
\end{theorem}

\begin{proof}
We endow $\Fcal$ with the uniform probability measure $\Pr(\,\cdot\,)=|\,\cdot\,|/|\Fcal|$.
Observe that $\Stab_\Gam(F)\neq\{\id\}$ if and only if there exists a nontrivial
$g\in\Gam$ with $g\cdot F=F$. Hence
\[
   \{F\in\Fcal:\Stab_\Gam(F)\neq\{\id\}\}
   = \bigcup_{g\in\Gam\setminus\{\id\}} \Fix(g),
\]
and by Lemma~\ref{lem:fix-g} each summand has probability
$\Pr(\Fix(g))=|\Fix(g)|/|\Fcal|\le q^{-(1-c)mq^n}$.
Applying the union bound over the $|\Gam|-1$ nontrivial elements of $\Gam$,
\begin{align*}
   \Pr\bigl(\Stab_\Gam(F)\neq\{\id\}\bigr)
   &\le \sum_{g\in\Gam\setminus\{\id\}} \Pr(\Fix(g)) \\
   &\le (|\Gam|-1)\cdot q^{-(1-c)mq^n}
   \le |\Gam|\cdot q^{-(1-c)mq^n}.
\end{align*}
To complete the proof, we bound $|\Gam|$. Recall that
\[
   |\GL(n,q)| = \prod_{i=0}^{n-1}(q^n-q^i) \le (q^n)^n = q^{n^2},
\]
so $|\AGL(U)| = q^n\cdot|\GL(n,q)| \le q^{n^2+n} = q^{O(n^2)}$.
Similarly, $|\AGL(W)| = q^{O(m^2)}$, and $|\Hom(U,W)|=q^{nm}$. Hence
\[
   |\Gam| = |\AGL(U)|\cdot|\AGL(W)|\cdot|\Hom(U,W)|
   \le q^{O(n^2+m^2+nm)} = q^{O((n+m)^2)},
\]
and the result follows.
\end{proof}

In the binary case ($q=2$), Lemma~\ref{lem:orbit-count} gives $(q+1)/(2q)=3/4$,
so for $m\ge 4$ we have $c=\max\{3/4,1-1/m\}=3/4$. The probability bound becomes
\[
   \Pr\bigl(\Stab_\Gam(F)\neq\{\id\}\bigr)
   \;\le\;
   2^{O((n+m)^2)} \cdot 2^{-(1/4)m2^n}
   \;=\;
   2^{-\Omega(m2^n)}.
\]
This probability decays super-exponentially in $n$, demonstrating that
vectorial Boolean functions with nontrivial EA-stabilizers are exponentially
rare. In other words, cryptographic S-boxes generically possess trivial EA-stabilizers:
the overwhelming majority of functions in $\Fcal$ have trivial stabilizer groups,
and only an exponentially negligible fraction admit any nontrivial element in their
EA-stabilizer.

\begin{example}[Concrete bounds for cryptographic parameters]\label{ex:concrete}
We evaluate the bound from Theorem~\ref{thm:trivial-stab} for parameters
commonly used in symmetric cryptography, i.e., $8$-bit S-boxes($n=m=8$):
We have
$\log_2|\Gam| \le 2(8^2+8) + 64 = 208$,
and the exponent in the decay is
$(1-c)m 2^n = \frac{1}{4}\cdot 8\cdot 256 = 512$.
Hence
\[
   \Pr\bigl(\Stab_\Gam(F)\neq\{\id\}\bigr)
   \;\le\; 2^{208-512} = 2^{-304}.
\]
This shows that a uniformly random $8$-bit S-box has a nontrivial
EA-stabilizer with probability at most $2^{-304}$, an astronomically
small quantity
Thus, for dimensions $n\ge 8$ (the regime relevant to modern block ciphers such as AES),
the trivial-stabilizer property holds with overwhelming probability.
\end{example}

\begin{remark}
Our upper bounds on $|\Fix(g)|$ and the resulting estimate
$\Pr[\Stab(F)\ne\{\id\}]\le q^{-\Omega(m q^n)}$
are obtained via worst-case analysis.
In principle, one could refine the argument by exploiting the
cycle structure of affine permutations on $U$ and $W$ and by computing
the rank distribution of $Q-I$ for $Q\in\GL(W)$.
Such refinements would slightly improve the constants in the exponent,
but would not affect the asymptotic type of the decay, and would require
a substantial amount of additional case analysis.
For the purposes of establishing asymptotic freeness, the present bounds
are entirely sufficient.
\end{remark}

\begin{corollary}[Asymptotic count of EA-classes]\label{cor:orbit-count}
Denote the set of EA-classes by $\Fcal/\Gam = \{\, \Gam\cdot F : F\in\Fcal \,\}$.
Then
\[
   \frac{|\Fcal/\Gam|}{|\Fcal|/|\Gam|}
   = 1 + o(1).
\]
In other words, the relative error between the true number of EA-classes
and the naive estimate $|\Fcal|/|\Gam|$ is $o(1)$.
\end{corollary}

\begin{proof}
Burnside's lemma gives
\[
   |\Fcal/\Gam|
   = \frac{1}{|\Gam|}\sum_{g\in \Gam}|\Fix(g)|.
\]
The identity contributes $|\Fcal|$.
Each nontrivial $g$ contributes at most $q^{c mq^n}$ by Lemma~\ref{lem:fix-g}.
Hence
\[
   |\Fcal/\Gam|
   = \frac{|\Fcal|}{|\Gam|}
     + O\!\left( |\Gam|\cdot q^{c mq^n}\right).
\]
The relative error is
\[
   \frac{|\Gam|\cdot q^{c mq^n}}{|\Fcal|/|\Gam|}
   = \frac{|\Gam|^2\cdot q^{c mq^n}}{|\Fcal|}
   = q^{-(1-c)mq^n + O((n+m)^2)}
   = o(1).
\]
To justify the exponent, recall from the proof of Theorem~\ref{thm:trivial-stab}
that $|\Gam| = q^{O((n+m)^2)}$, hence $|\Gam|^2 = q^{O((n+m)^2)}$.
Thus the relative error tends to zero.
\end{proof}

For the binary case ($q=2$), we have
\[
   \frac{|\Fcal/\Gam|\cdot |\Gam|}{2^{m2^n}}
   = 1 + o(1),
\]
where $|\Gam|=|\AGL(n,2)|\cdot|\AGL(m,2)|\cdot 2^{nm}$.
Consequently, the contribution of functions
with nontrivial EA-stabilizers to the total count is exponentially negligible:
their number is exponentially smaller than the full space.
This observation is particularly relevant for structured searches in the design
of cryptographic S-boxes.

\begin{proposition}[Upper bounds on collision probabilities]\label{cor:collision-upper}
For two independently and uniformly random functions $F,G\in\Fcal$, the probability
that they are equivalent satisfies:
\begin{enumerate}
\item \emph{EA-equivalence:}
\[
\Pr(F\sim_{\mathrm{EA}} G)
\;\le\;
\frac{|\Gam|}{|\Fcal|}
\;=\;
q^{-mq^n + O((n+m)^2)}
\;=\;
q^{-\Omega(q^n)}.
\]
\item \emph{CCZ-equivalence:}
\[
\Pr(F\sim_{\mathrm{CCZ}} G)
\;\le\;
\frac{|\AGL(U \times W)|}{|\Fcal|}
\;=\;
q^{-mq^n + O((n+m)^2)}
\;=\;
q^{-\Omega(q^n)}.
\]
\end{enumerate}
\end{proposition}

\begin{proof}
For both equivalence relations, the collision probability can be written as
\[
\Pr(F \sim G)
= \sum_{i=1}^r \Bigl( \frac{|O_i|}{|\Fcal|} \Bigr)^2
\le \sum_{i=1}^r \frac{|O_i|}{|\Fcal|} \cdot \frac{\max_j |O_j|}{|\Fcal|}
= \frac{\max_j |O_j|}{|\Fcal|},
\]
where $O_1,\dots,O_r$ are the orbits. By the orbit-stabilizer theorem, the maximum
orbit size is bounded by the group order. For EA-equivalence,
$\max_j |O_j| \le |\Gam|=q^{O((n+m)^2)}$, and for CCZ-equivalence,
$\max_j |O_j| \le |\AGL(U \times W)|=q^{(n+m)^2+O(n+m)}$. Combined with
$|\Fcal| = q^{mq^n}$, this gives the stated bounds.
\end{proof}

We emphasise that these upper bounds rely only on the orbit-stabilizer theorem and
the elementary estimates on the group orders; they do \emph{not} require the
trivial-stabilizer result of Theorem~\ref{thm:trivial-stab}.

The upper bounds in Proposition~\ref{cor:collision-upper} follow directly from the
orbit-stabilizer theorem and hold for both EA- and CCZ-equivalence over arbitrary finite
fields. However, using our result on trivial stabilizers (Theorem~\ref{thm:trivial-stab}),
we can show that the EA upper bound is asymptotically tight:

\begin{corollary}[Asymptotic collision probability for EA-equivalence]\label{cor:collision-ea-asymptotic}
For two independently and uniformly random functions $F,G\in\Fcal$,
the probability that $F$ and $G$ are EA-equivalent satisfies
\[
\Pr(F\sim_{\mathrm{EA}} G)
= \frac{|\Gam|}{|\Fcal|}\,(1+o(1)).
\]
\end{corollary}

\begin{proof}
Let $O_1,\dots,O_r$ be the EA-orbits. The collision probability is
\begin{equation}\label{eq:collision-ea}
\Pr(F\sim_{\mathrm{EA}} G)
= \sum_{i=1}^r \Bigl( \frac{|O_i|}{|\Fcal|} \Bigr)^2
= \frac{1}{|\Fcal|^2}\sum_{i=1}^r |O_i|^2.
\end{equation}

Let $\mathcal O_{\mathrm{triv}}$ be the set of orbits with trivial EA-stabilizer
and $\mathcal O_{\mathrm{nt}}$ the set of the remaining orbits.
By Theorem~\ref{thm:trivial-stab}, for every $\varepsilon>0$ and all sufficiently large $n$,
\begin{equation}\label{eq:measure-trivial}
|\{F\in\Fcal: \Stab_\Gam(F)=\{\id\}\}|
\ge (1-\varepsilon)|\Fcal|.
\end{equation}
Hence
\[
\sum_{O_i\in \mathcal O_{\mathrm{triv}}} |O_i|
= |\{F:\Stab_\Gam(F)=\{\id\}\}|
\ge (1-\varepsilon)|\Fcal|.
\]
Since each such orbit has size $|\Gam|$, this implies
\[
|\mathcal O_{\mathrm{triv}}|
\ge \frac{(1-\varepsilon)|\Fcal|}{|\Gam|}.
\]

The contribution of the trivial-stabilizer orbits to \eqref{eq:collision-ea} is therefore
\[
\frac{1}{|\Fcal|^2}
\sum_{O_i\in\mathcal O_{\mathrm{triv}}} |O_i|^2
=
\frac{|\mathcal O_{\mathrm{triv}}|\cdot |\Gam|^2}{|\Fcal|^2}
\ge (1-\varepsilon)\frac{|\Gam|}{|\Fcal|}.
\]

For the nontrivial-stabilizer orbits, since
$|O_i|\le |\Gam|$ and
$\sum_{O_i\in\mathcal O_{\mathrm{nt}}} |O_i| \le \varepsilon |\Fcal|$
by \eqref{eq:measure-trivial}, we have
\[
\frac{1}{|\Fcal|^2}\sum_{O_i\in\mathcal O_{\mathrm{nt}}} |O_i|^2
\le
\varepsilon\frac{|\Gam|}{|\Fcal|}.
\]

Combining both contributions with the upper bound from Proposition~\ref{cor:collision-upper} yields
\[
(1-\varepsilon)\frac{|\Gam|}{|\Fcal|}
\;\le\;
\Pr(F\sim_{\mathrm{EA}} G)
\;\le\;
\frac{|\Gam|}{|\Fcal|}.
\]
Since $\varepsilon>0$ is arbitrary, the result follows.
\end{proof}

This asymptotic equality validates random sampling strategies for
searching cryptographic primitives. For two independently sampled functions,
the probability of landing in the same EA-equivalence class is negligibly
small. In the binary case, the collision probability is $2^{-\Omega(2^n)}$,
which means that random generation of vectorial Boolean functions produces
EA-inequivalent candidates with overwhelming probability. This is particularly
useful for exhaustive searches of functions with special cryptographic
properties such as almost perfect nonlinear (APN) and almost bent (AB)
functions: researchers can sample the function space randomly without concern
for redundancy modulo EA-equivalence, as collisions are exponentially rare.

\section{Conclusion}\label{sec:conclusion}

We have established that asymptotically almost all vectorial functions over finite fields
have trivial EA-stabilizers. This fundamental result resolves the two questions posed in the
introduction.

\begin{enumerate}
\item \emph{Classification:} The number of EA-equivalence classes is asymptotically
equal to the naive estimate (the total number of functions divided by the size of the
EA-group) with vanishing relative error. This shows that the EA group action is asymptotically
free: almost every function has a trivial EA-stabilizer. For cryptographic applications, this
means that the vast majority of vectorial functions lie in distinct EA-classes, and functions
with nontrivial EA-stabilizers form an exponentially small subset.

\item \emph{Random sampling:} Together with upper bounds on collision probabilities,
our results establish that for EA-equivalence, two independently sampled functions are equivalent
with super-exponentially small probability. This validates random sampling strategies for
searching cryptographic primitives such as APN or AB functions: exhaustive random searches
produce EA-inequivalent candidates with overwhelming probability, eliminating concerns about
redundancy modulo EA-equivalence.
\end{enumerate}

While we have proved upper bounds on collision probabilities for CCZ-equivalence, the question
of whether most functions have trivial CCZ-stabilizers remains open. Establishing analogous results for CCZ-equivalence would require new techniques, as the
graph-based action is more complex than the EA group action.

\section*{Acknowledgements}

This work was supported in part by JSPS KAKENHI Grant Number JP25K17290.

\bibliographystyle{plain}
\bibliography{references}

\end{document}